\documentclass[a4paper,10pt,reqno, english]{amsart}

\usepackage[utf8]{inputenc}
\usepackage[T1]{fontenc}
\usepackage{amsmath,amsthm}
\usepackage{amsfonts,amssymb,enumerate}
\usepackage{url,paralist}
\usepackage{mathtools,color}
\usepackage[colorlinks=true,urlcolor=blue,linkcolor=red,citecolor=magenta,hypertexnames=false]{hyperref}

\usepackage{enumerate}
\usepackage{anysize}

\theoremstyle{plain}
\newtheorem{theorem}{Theorem}[section]
\newtheorem{lemma}[theorem]{Lemma}

\newtheorem{corollary}[theorem]{Corollary}

\newtheorem{proposition}[theorem]{Proposition}

\newtheorem*{theorem*}{Theorem}

\newtheorem*{claim*}{Claim}
\newtheorem*{lemma*}{Lemma}

\theoremstyle{definition}
\newtheorem{definition}[theorem]{Definition}

\newtheorem{example}[theorem]{Example}
\newtheorem{examples}[theorem]{Examples}

\newcommand{\Z}{\mathbb{Z}}
\newcommand\sym{\mathfrak S}
\newcommand\stabsym{(\sym_k)_{(\mathfrak{c}_1,\dots,\mathfrak{c}_k)}}

\newcommand{\id}{\operatorname{id}}
\newcommand{\relint}{\operatorname{relint}}

\newcommand{\defn}[1]{\emph{#1}} %
\newcommand\nonumberthis{}

\begin{document}

\title[Counting Gray codes]{Counting Gray codes for an improved upper bound of the Grünbaum--Hadwiger--Ramos problem}


%
\author[Kliem]{Jonathan Kliem}
\address{Institut f\" ur Mathematik, FU Berlin, Arnimallee 2, 14195 Berlin, Germany}
\email{jonathan.kliem@fu-berlin.de}
\thanks{J.K.~receives funding by the Deutsche Forschungsgemeinschaft DFG under Germany's Excellence Strategy – The Berlin Mathematics Research Center MATH+ (EXC-2046/1, project ID: 390685689).}

\date{\today}

\begin{abstract}
    We give an improved upper bound for the Grünbaum--Hadwiger--Ramos problem:
    Let $d,n,k \in \mathbb{N}$ such that $d \geq 2^n(1+2^{k-1})$.
    Given $2^{n+1}$ masses on $\mathbb{R}^d$,
    there exist $k$ hyperplanes in $\mathbb{R}^d$ that partition it into $2^k$ sets of equal size with respect to all measures.
    This is an improvement to the previous bound $d \geq 2^{n + k}$ by Mani-Levitska, Vre\'cica \& \v{Z}ivaljevi\'c in 2006.

    This is achieved by classifying the number of certain Gray code patterns modulo 2.
    The reduction was developed by Blagojević, Frick, Haase \& Ziegler in 2016.
    It utilizes the group action of the symmetric group $(\Z/2)^k \rtimes \sym_k$ of $k$ oriented hyperplanes.
    If we restrict to the subgroup $(\Z/2)^k$ as Mani-Levitska~et~al.~we retrieve their bound.
\end{abstract}

\maketitle
\tableofcontents

\section{Introduction}
\label{sec : Introduction}

Consider the $k$-dimensional hypercube $[0,1]^k$.
Its graph has several Hamiltonian paths that is, paths of vertices and edges that visit every vertex exactly once.
A collection of $0,1$-vectors representing such a Hamiltonian path is known as $k$-bit Gray code.
The simplest and for our purposes most important such path is the standard $k$-bit Gray code, which
is defined inductively as traversing the front facet by the standard $(k-1)$-bit Gray code and the back facet by the standard $(k-1)$-bit Gray code in reverse order.

The numbers of $k$-bit Gray codes with marked starting node for $k=1,\dots,5$ are $1, 2, 18, 5712$ and $5859364320$\footnote{\url{https://oeis.org/A003043}}.
For $k \geq 6$ this number is not known.
Invariants of a Gray code are the transition counts $c_1,\dots,c_k$, where $c_i$ counts the number of edges parallel to the $i$-th standard base vector for each $i=1,\dots,k$.
We can classify for which transition counts the number of Gray codes with fixed starting node is odd:

\begin{theorem}\label{Thm:ShuffleCount}
    The standard $k$-bit Gray code has transition counts $(2^{k-1}, 2^{k-2},\dots,4,2,1)$ and it is the only Gray code (up to choosing a starting node) with those transition counts.
    Up to permutation those are the only transition counts that are represented by an odd number of Gray codes.
\end{theorem}

Given a tuple of $j > 1$ $k$-bit Gray codes.
One can add their transition counts.
We try to minimize $\max(\mathfrak{c}_1,\dots,\mathfrak{c}_k)$ such that the number of $j$-tuples of $k$-bit Gray codes with fixed starting node and sum of transition counts equal to $(\mathfrak{c}_1,\dots,\mathfrak{c}_k)$ is odd.
If $j = 2^n + r$ with $0 \leq r < 2^n$ we find that this value is exactly $2^{k-1}2^n + r$;
see Theorem~\ref{Thm:OddEquipartingMatrices}~\eqref{Thm:OddEquipartingMatricesExact}.

Now the symmetric group of the $k$-dimensional hypercube acts on Hamiltonian paths not just by changing the starting node, but also by permuting coordinates.
It acts on tuples of Gray codes by diagonal action.
Hence, the number of $2$-tuples of $4$-bit Gray codes with transition counts $(17, 17, 8, 4)$ is even, as those transition counts are invariant with respect to transposing the first two coordinates.
We will see however, that this number is not divisible by four; see Proposition~\ref{Prop:ExOddCount}.

We can relax our condition to minimize $\max(\mathfrak{c}_1,\dots,\mathfrak{c}_k)$ such that the number of orbits of those $j$-tuples of Gray codes is odd.
Theorem~\ref{Thm:OddEquipartingMatrices} establishes that if $j = 2^n + r$ with $0 < r \leq 2^n$, this value is exactly $2^{k-1}2^n + r$;
see Theorem~\ref{Thm:OddEquipartingMatrices}~\eqref{Thm:OddEquipartingMatricesEquivariant}.
This differs from the above when $j = 2^n + 2^n$ and this value is $(1 + 2^{k-1})2^n$ instead of $2^k2^n$.
Using methods developed by Blagojević, Frick, Haase, \& Ziegler in~\cite{BFHZ2016} our combinatorial observations lead to a new bound for the Grünbaum--Hadwiger--Ramos problem:

\bigskip

A mass on $\mathbb{R}^d$ is a finite Borel measure that vanishes on every affine hyperplane.
Hugo Steinhaus conjectured that any $d$ masses on $\mathbb{R}^d$ can be simultaneously bisected by an appropriate (affine) hyperplane~\cite{HistoryHamSandwich}.
This is now known as the Ham--Sandwich theorem.
Several authors have studied the following generalization:

An arrangment of $k$ hyperplanes in $\mathbb{R}^d$ partitions its complement into $2^k$ (possibly empty) orthants.
A mass $\mu$ is equiparted by this arrangment, if each orthant has measure $\frac{1}{2^k} \mu(\mathbb{R}^d)$.
The Grünbaum--Hadwiger--Ramos problem is to find the minimal $d = \Delta(j, k)$ such that for any $j$ masses on $\mathbb{R}^d$ there exists an arrangment of $k$ hyperplanes simultaneously equiparting all masses.
For a brief history of this problem, we refer the reader to~\cite{BFHZ2016}.

The lower bound developed by Avis~\cite{Avis} and Ramos~\cite{Ramos} is
\[
    \Delta(j, k) \geq \frac{2^k -1}{k}j.
\]
It is conjectured to be tight and for $k=1$ this is the case by the Ham--Sandwich theorem.
However, even for $k=2$ exact bounds are elusive:
In this case it is shown to be tight for $j = 2^n -1, 2^n, 2^n + 1$ by Blagojević, Frick, Haase, \& Ziegler~\cite[Thm.~1.5]{BFHZ2016}.
Partial results were known before, see~\cite{BFHZ2016}.

For $k>2$ the only tight bounds are $\Delta(1, 3) = 3$ by Hadwiger~\cite{Hadwiger} and $\Delta(2,3)=5$, $\Delta(4,3)=10$ by~\cite{BFHZ2016}.
The reduction of Hadwiger and Ramos $\Delta(j, k) \leq \Delta(2j, k-1)$ has been used to give (non-optimal) bounds $\Delta(1, 4) \leq 5$ and $\Delta(2, 4) \leq 10$ and $\Delta(1, 5) \leq 10$; see~\cite{BFHZ2016}.

We complete the list of all previous known bounds with the general upper bound
\[
    \Delta(j,k) \leq j + (2^{k-1} - 1)2^{\lfloor \log_2 j \rfloor}
\]
or equivalently
\[
    \Delta(2^n+r,k) \leq 2^{n+k-1}+r
    \quad \text{for } 0 \leq r < 2^n
\]
by Mani-Levitska, Vrećica \& Živaljević \cite{Mani}.
We will use the methods develeoped in~\cite{BFHZ2016} and our combinatorial results briefly described above to generalize this bound:

\begin{theorem}\label{Thm:TopMainStatement}
    Let $j, k \geq 1$ be integers. It holds that
    \[
        \Delta(j,k) \leq \Big\lceil j + (2^{k-1} - 1)2^{\lfloor \log_2 (j - \frac{1}{2}) \rfloor} \Big\rceil.
    \]
    Equivalently, for integers $j = 2^n + r$, $n \geq 0$, and $1 \leq r \leq 2^n$ it holds that
    \[
        \Delta(2^n+r,k) \leq 2^{n+k-1}+r \quad \text{and} \quad
        \Delta(1,k) \leq \begin{cases}
            2^{k - 2} + 1, & \text{if } k \geq 2\\
            1, & \text{if } k = 1.
        \end{cases}
    \]
\end{theorem}

Note that \cite[Thm.~1.5~(iii)]{BFHZ2016} provides a better and tight bound for $\Delta(2^n +1 , 2)$, which cannot be generalized for $k > 2$.
Other than this, the above theorem collects all currently known upper bounds.
To illustrate the improvements, we give a few examples:

\begin{corollary}
    \
    \begin{enumerate}
        \item $\Delta(2^{n+1}, 2) \leq 2^{n+1} + 2^n = 3 \cdot 2^{n}$, shown before in~\cite{BFHZ2016}.
        \item $\Delta(2, 3) \leq 5$, shown before in~\cite{BFHZ2016}.
        \item $\Delta(4, 3) \leq 10$, shown before in~\cite{BFHZ2016}.
        \item $\Delta(8, 3) \leq 20$, where the previous best bound is $32$; see~\cite{Mani}.
        \item $\Delta(16, 3) \leq 40$, where the previous best bound is~$64$; see~\cite{Mani}.
        \item $\Delta(1, 4) \leq 5$, shown before in~\cite{BFHZ2016}.
        \item $\Delta(2, 4) \leq 9$, where the previous best bound is $10$; see~\cite{BFHZ2016}.
        \item $\Delta(4, 4) \leq 18$, where the previous best bound is $32$; see~\cite{Mani}.
        \item $\Delta(8, 4) \leq 36$, where the previous best bound is $64$; see~\cite{Mani}.
        \item $\Delta(1, 5) \leq 9$, where the previous best bound is $10$; see~\cite{BFHZ2016}.
        \item $\Delta(2, 5) \leq 17$, where the previous best bound is $32$; see~\cite{Mani}.
        \item $\Delta(4, 5) \leq 34$, where the previous best bound is $64$; see~\cite{Mani}.
        \item $\Delta(2^{n+1}, k) \leq (2^k+1)\cdot 2^n$, where the previous best bound is $2^{k+1}\cdot 2^n$; see~\cite{Mani}.
    \end{enumerate}
\end{corollary}

If $j$ is not a power of two, our bound coincides with the bound in~\cite{Mani} and we provide an alternative proof of their result.
Their bound was derived from a Fadell–Husseini index calculation using the product scheme, parametrizing arrangements of $k$ hyperplanes by $(S^d)^k$.
Our proof relies on equivariant obstruction theory with the join scheme, which parametrizes arrangements of $k$ hyperplanes by $(S^d)^{*k}$.
(See \cite[Section~1.1]{BFHZ2016} for more detail.)
Using the join scheme instead of the product scheme, allows to use the full symmetric group $(\Z/2)^k \rtimes \sym_k$ of $k$ oriented hyperplanes.
If we restrict to the subgroup $(\Z/2)^k$ we retrieve exactly the result by~\cite{Mani}.

To conclude the introduction, we remark that for $j=1$ and $k > 1$ our new bound
\[
    \Delta(1, k) \leq 2^{k - 2} + 1
\]
is already implied by the reduction of Hadwiger and Ramos $\Delta(1, k) \leq \Delta(2, k-1)$ and our new bound
\[
    \Delta(2, k-1) \leq 2^{k-2} + 1.
\]

\section{Counting Gray codes}

Blagojević, Frick, Haase and Ziegler \cite{BFHZ2016} have approached the Grünbaum--Hadwiger--Ramos problem with obstruction theory and
developed a reduction of this problem to a combinatorial problem.
The topological problem can be solved by counting equivalence classes of certain $0,1$-matrices modulo $2$ with number of rows equal to $k$.
They have classified this parity for the matrices corresponding to $k = 2$.
For $j \leq 9$, $k = 3$ and $j \leq 2$, $k = 4$ they have counted some of them with the help of a computer to obtain new bounds for $\Delta(2,3)$ and $\Delta(4,3)$,
see \cite[Table 1]{BFHZ2016}.
By the reduction of Hadwiger and Ramos they could derive bounds for $\Delta(1,4)$ and $\Delta(2,4)$.

Instead of listing all those matrices, one can use an iterator to count more matrix classes.
This helped to understand the underlying structure:
For $j > 2$ we provide a full classification, when the number of equivalence classes is odd.

\subsection{Gray codes}
We will use notation from Knuth~\cite[pp.~292--294]{Knuth} but start indexing with $1$ instead of $0$.

\begin{definition}[{\cite[Sec.~1.3]{BFHZ2016}}]
A $k$-bit Gray code is a $k \times 2^k$ binary matrix containing all column vectors in $\{0,1\}^k$ such that two consecutive vectors differ in only one entry.
\end{definition}

A Gray code can be seen as a Hamiltonian path on the edge graph of a hypercube.
This is a path visiting each vertex.
Knuth restricts to Hamiltonian paths that can be completed to a cycle.
We will not restrict to this case.

\begin{definition}
    The \defn{standard $k$-bit Gray code} is a path obtained inductively.
    The standard $1$-bit Gray code is $\left(\begin{pmatrix}0\end{pmatrix}, \begin{pmatrix}1\end{pmatrix}\right)$.
    For $k > 1$ the standard $k$-bit Gray code is given by
    \begin{itemize}
        \item each vector of the standard $(k-1)$-bit Gray code with $0$ appended,
        \item in reverse order each vector of the standard $(k-1)$-bit Gray code with $1$ appended.
    \end{itemize}
\end{definition}
Thus the standard $k$-bit Gray code traverses the front facet by the standard $(k-1)$-bit Gray code and then in reverse order the back facet.

\begin{example}
    The standard $3$-bit Gray code is
    \[
        G = \begin{pmatrix}
            0 & 1 & 1 & 0 & 0 & 1 & 1 & 0\\
            0 & 0 & 1 & 1 & 1 & 1 & 0 & 0\\
            0 & 0 & 0 & 0 & 1 & 1 & 1 & 1
        \end{pmatrix}.
    \]
\end{example}
We will see later that for $j > 2$ we only need to consider standard $k$-bit Gray codes.

\begin{definition}
    Let $G = (g_1,\dots,g_{2^k})$ be a $k$-bit Gray code. Then the \defn{delta sequence of} $G$
    \[
        \delta(G) := (\delta_1(G),\dots,\delta_{2^{k}-1}(G))
    \]
    is defined by $\delta_i(G)$ being the row of the bit change from $g_i$ to $g_{i+1}$ for each $i = 1,\dots, 2^{k} - 1$.

    A \defn{$k$-bit delta sequence} is a vector, which is the delta sequence of some $k$-bit Gray code.
\end{definition}

\begin{example}
    The standard $3$-bit Gray code has delta sequence $(1,2,1,3,1,2,1)$.
    The standard $4$-bit Gray code has delta sequence $(1,2,1,3,1,2,1,4,1,2,1,3,1,2,1)$.
    In general, $\delta_i$ of the standard $k$-bit Gray code is equal to one plus the dyadic order of $i$.
\end{example}

Let $D$ be a $k$-bit delta seqence.
For each choice of a first column $g_1 \in \{0,1\}^k$ there exists exactly one Gray code $G$ with first column $g_1$ and $\delta(G) = D$.

The symmetric group of the $k$-dimensional hypercube,
$(\Z/2)^k \rtimes \sym_k =: \sym_k^{\pm}$,
acts on the set of all Gray codes by permuting rows and inverting all bits in one row.
The group $\sym_k^{\pm}$ also acts on $[k]^{2^k-1}$ by
\[
    ((\beta_1,\dots,\beta_k) \rtimes \tau) \cdot (d_1,\dots,d_{2^k - 1}) = \big(\tau(d_1), \dots, \tau(d_{2^k -1})\big).
\]
With this action, the map $\delta$ from all $k$-bit Gray codes to $[k]^{2^k-1}$ is $\sym_k^{\pm}$-equivariant.

Note that $\sym_k^{\pm}$ acts freely on $k$-bit Gray codes, that is $\sigma \in \sym_k^{\pm}$ and $G$ a gray code with $\sigma \cdot G = G$ implies that $\sigma$ is the unit element.
On the other hand, $\sym_k$ acts freely on $k$-bit delta sequences (but not on $[k]^{2^k-1}$) and $(\Z/2)^k$ acts trivially on $k$-bit delta sequences.
So $(\Z/2)^k$ acts on Gray codes by changing the first column while stabilizing the delta sequence.

\begin{definition}
    Let $D$ be a $k$-bit delta sequence.
    The \defn{transition counts}
    \[
        c(D) := (c_1(D),\dots,c_k(D))
    \]
    are defined by letting $c_i(D)$ be the number of $i$'s in $D$, for every $i = 1,\dots,k$.

    The transition counts of a Gray code are the transition counts of its delta sequence.
\end{definition}
\begin{example}
    The transition counts of the standard $k$-bit Gray code are $(2^{k-1}, 2^{k-2},\dots, 4,2,1)$.
\end{example}

$\sym_k^{\pm}$ acts on transition counts by trivial action of $(\Z/2)^k$ and by $\sym_k$ permuting coordinates.
With this action, taking the transition counts is $\sym_k^\pm$-equivariant.

\begin{lemma}\label{Lem:TransitionCounts}
    Let $D = (d_1,\dots,d_{2^k -1})$ be a $k$-bit delta sequence.
    Suppose for some $0 \leq i < k$ it holds that
    \[
        c_k(D) = 1, \quad c_{k-1}(D) = 2, \quad \dots, \quad c_{k-i+1}(D) = 2^{i-1}.
    \]
    Then, $c_{k-i}(D) \geq 2^{i}$.
\end{lemma}
\begin{proof}
    We induct on $i$ for all $k$.
    For $i = 0$ the statement holds as $c_k(D) \geq 1$.
    Suppose now that $i > 0$ and
    \[
        c_k(D) = 1, \quad c_{k-1}(D) = 2, \quad \dots, \quad c_{k-i+1}(D) = 2^{i-1}.
    \]
    As $c_k(D) = 1$ it follows that the $k$-th bit must change at the central position.
    Then, $\overline{D}_1 := (d_1,\dots,d_{2^{k-1}-1})$ and $\overline{D}_2 := (d_{2^{k-1}+1},\dots,d_{2^{k}-1})$ are both $(k-1)$-bit delta sequences such that
    \[
        c_{k-1}(\overline{D}_1) + c_{k-1}(\overline{D}_2) = 2, \quad c_{k-2}(\overline{D}_1) + c_{k-2}(\overline{D}_2) = 4, \quad \dots, \quad c_{k-i+1}(\overline{D}_1) + c_{k-i+1}(\overline{D}_2) = 2^{i-1},
    \]
    which implies by induction that
    \[
        c_{k-1}(\overline{D}_1) = c_{k-1}(\overline{D}_2) = 1, \quad c_{k-2}(\overline{D}_1) = c_{k-2}(\overline{D}_2) = 2, \quad \dots, \quad c_{k-i+1}(\overline{D}_1) = c_{k-i+1}(\overline{D}_2) = 2^{i-2},
    \]
    and therefore $c_{k-i}(\overline{D}_1) \geq 2^{i-1}$ and $c_{k-i}(\overline{D}_2) \geq 2^{i-1}$.
\end{proof}

\begin{definition}
    Let $D$ be a $k$-bit delta sequence.
    The \defn{shuffle count} of $D$ is the number of delta sequences $\widetilde{D}$ with $c(D) = c(\widetilde{D})$.
\end{definition}

We will now proof \ref{Thm:ShuffleCount} by the following two lemmas:

\begin{lemma}\label{Lem:ShuffleCountStandard}
    Let $D$ be the delta sequence of the standard $k$-bit Gray code.
    The shuffle count of $D$ is $1$.
\end{lemma}
\begin{proof}
    For $k = 1$ the statement is clear.
    Let $\widetilde{D} = (\tilde{d}_1,\dots, \tilde{d}_{2^k -1})$ be a $k$-bit delta sequence with $c(\widetilde{D}) = c(D)$.
    The equality $c_k(\widetilde{D}) = c_k(D) = 1$ implies that the only $k$-bit change must be at central position and $\tilde{d}_{2^{k-1}} = k$ as desired.
    By Lemma~\ref{Lem:TransitionCounts}
    it follows that
    $\overline{D}_1 := (\tilde{d}_1,\dots,\tilde{d}_{2^{k-1}-1})$ and $\overline{D}_2 := (\tilde{d}_{2^{k-1}+1},\dots,\tilde{d}_{2^{k}-1})$ are both $(k-1)$-bit delta sequences such that
    $c(\overline{D}_1) = c(\overline{D}_2) = (2^{k-2}, 2^{k-3}, \dots, 1)$.
    By induction this implies that $\overline{D}_1$ and $\overline{D}_2$ are the delta sequence of the standard $(k-1)$-bit Gray code.
\end{proof}

\begin{lemma}\label{Lem:ShuffleCountNonStandard}
    Let $D = (d_1,\dots,d_{2^k -1})$ be a $k$-bit delta sequence.
    If the shuffle count of $D$ is odd, then $D$ lies in the $\sym_k^\pm$-orbit of the delta sequence of the standard $k$-bit Gray code.
\end{lemma}
\begin{proof}
    Denote by $D^R$ the sequence of $D$ in reversed order.
    If $D$ is a delta sequence, then so is $D^R$.
    Hence the parity of the shuffle count is equal to the parity of the number of symmetric delta sequences $\widetilde{D}$ with $c(\widetilde{D}) = c(D)$.
    If the shuffle count of $D$ is odd, then w.l.o.g. $D$ itself is symmetric.
    Let $G$ be a Gray code with $\delta(G) = D$.
    As the statement is invariant under the group action of $\sym_k^\pm$, we may assume that the central entry of $D$ is $k$.

    As $D$ is symmetric, it follows that for each $i = 1,\dots,2^{k-1}$ the columns $G_i$ and $G_{2^k - i+1}$ differ exactly in the row $k$.
    This implies that $c_k(D) = 1$ and $\overline{D}_1 := (d_1,\dots,d_{2^{k-1}-1})$ and $\overline{D}_2 := (d_{2^{k-1} + 1}, \dots, d_{2^k - 1}) = \overline{D}_1^R$ are $(k-1)$-bit delta sequences.
    As the shuffle count of $D$ is odd, we may assume that $\overline{D}_1$ itself has odd shuffle count.
    By induction this implies that $\overline{D}_1$ lies in the $\sym_{k-1}^\pm$-orbit of the standard $(k-1)$-bit Gray code and then $D$ lies in the $\sym_k^\pm$-orbit of the standard $k$-bit Gray code.
\end{proof}

\subsection{Equiparting matrices}

We generalize~\cite[Def.~1.1]{BFHZ2016} as follows:

\begin{definition}
    Let $\mathfrak{c}_1 + \dots + \mathfrak{c}_k = j (2^k -1)$.
    A binary matrix $\mathcal{G}$ of size $k \times j2^k$ is a $(\mathfrak{c}_1,\dots,\mathfrak{c}_k)$-equiparting matrix if
\begin{enumerate}
 \item $\mathcal{G} = (G_1,\dots,G_j)$ for Gray codes $G_1,\dots,G_j$ with the property that the last column of $G_i$ is equal to the first column of $G_{i+1}$ for $1 \leq i < j$; and
 \item the transition counts of $\mathcal{G}$ are given by
     \[
         c(\mathcal{G}) = c(G_1) + \dots + c(G_j) = (\mathfrak{c}_1, \dots, \mathfrak{c}_k).
     \]
\end{enumerate}
\end{definition}
If $\mathfrak{c}_1 > \mathfrak{c}_2 = \mathfrak{c}_3 = \dots = \mathfrak{c}_k$, we have a $(\mathfrak{c}_1 - \mathfrak{c}_2)$-equiparting matrix from \cite[Def.~1.1]{BFHZ2016}.

Let $\sym_k$ act on $(\mathfrak{c}_1,\dots,\mathfrak{c}_k)$ by permutation and let $(\sym_k)_{(\mathfrak{c}_1,\dots,\mathfrak{c}_k)}$ denote the stabilizer subgroup.
Then
$\stabsym^\pm$
acts on
$(\mathfrak{c}_1,\dots,\mathfrak{c}_k)$-equiparting matrices
by acting on all Gray codes simultaneously.
Note that $\stabsym^\pm$ acts freely on
$(\mathfrak{c}_1,\dots,\mathfrak{c}_k)$-equiparting matrices.

\begin{definition}
    \
    \begin{enumerate}
        \item
            We say that two $(\mathfrak{c}_1,\dots,\mathfrak{c}_k)$-equiparting matrices $\mathcal{G}$ and $\mathcal{G}'$ are \defn{isomorphic},
            if they are in one $(\Z/2)^k$-orbit.
        \item
            We say that two $(\mathfrak{c}_1,\dots,\mathfrak{c}_k)$-equiparting matrices $\mathcal{G}$ and $\mathcal{G}'$ are \defn{equivalent},
            if they are in one $\stabsym^\pm$-orbit.
    \end{enumerate}
\end{definition}
\begin{definition}
    Let $j \geq 1$, $k \geq 2$ be integers.
    \begin{enumerate}
        \item
            Consider the set $\mathcal{I}(j,k)$ of all tuples of integers $(\mathfrak{c}_1, \dots, \mathfrak{c}_k)$ with $\mathfrak{c}_1 + \dots + \mathfrak{c}_k = j(2^k -1)$ such that the number of non-isomorphic $(\mathfrak{c}_1,\dots,\mathfrak{c}_k)$-equiparting matrices is odd.

            Denote by $I(j,k)$ the minimum of $\max\left(\mathfrak{c}_1,\dots,\mathfrak{c}_k\right)$ for all $(\mathfrak{c}_1,\dots,\mathfrak{c}_k) \in \mathcal{I}(j,k)$.
        \item
            Consider the set $\mathcal{E}(j,k)$ of all tuples of integers $(\mathfrak{c}_1, \dots, \mathfrak{c}_k)$ with $\mathfrak{c}_1 + \dots + \mathfrak{c}_k = j(2^k -1)$ such that the number of non-equivalent $(\mathfrak{c}_1,\dots,\mathfrak{c}_k)$-equiparting matrices is odd().

            Denote by $E(j,k)$ the minimum of $\max\left(\mathfrak{c}_1,\dots,\mathfrak{c}_k\right)$ for all $(\mathfrak{c}_1,\dots,\mathfrak{c}_k) \in \mathcal{E}(j,k)$.
    \end{enumerate}
\end{definition}

Clearly, $E(j,k) \leq I(j,k)$.
We obtain the following classification of the parity of $(\mathfrak{c}_1, \dots, \mathfrak{c}_k)$-equiparting matrices:
\begin{theorem}\label{Thm:OddEquipartingMatrices}
    Let $j \geq 1$, $k \geq 2$ be integers.

    \begin{enumerate}
        \item\label{Thm:OddEquipartingMatricesExact}
            Then
            \[
                I(2^n + r,k) = 2^{n+k-1} + r
            \]
            for all $0 \leq r < 2^n$ and $n \geq 0$.
        \item\label{Thm:OddEquipartingMatricesEquivariant}
            Then
            \[
                E(1, k) \leq 2^{k-2} + 1
            \]
            and
            \[
                E(2^n + r,k) = 2^{n+k-1} + r
            \]
            for all $1 \leq r \leq 2^n$ and $n \geq 0$.
    \end{enumerate}
\end{theorem}

Proposition~\ref{Prop:MainReduction} establishes that $\Delta(j, k) \leq I(j,k)$ and $\Delta(j,k) \leq E(j,k)$ and concludes our proof of Theorem~\ref{Thm:TopMainStatement}.
This reduction has basically been shown in \cite{BFHZ2016} and we generalize it in Section~\ref{Sec:Reduction}.
Note that these bounds only differ in the cases
\[
    E(1, k) \leq 2^{k-2} + 1, \quad I(1, k) = 2^{k-1}
\]
and
\[
    E(2^n + 2^n, k) = 2^{n+k-1} + 2^n, \quad I(2^{n+1}, k) = 2^{n+k}.
\]

The value of $I(j, k)$ that is obtained by only considering the $(\Z/2)^k$-action is exactly the bound by Mani-Levitska, Vrećica \& Živaljević~\cite{Mani}.
This seems natural as they also have restricted their attention to the $(\Z/2)^k$-action \cite[Proof of Thm.~38]{Mani}.
The value of $E(j, k)$ is the improved bound by considering the full $\sym_k^\pm$-action with the methods developed in~\cite{BFHZ2016}.

The exact value of $E(1, k)$ remains unknown as we do not understand the precise structure of non-equivalent Gray codes (equiparting matrices for $j=1$).
It has been verified to be tight by computation for $k \leq 5$.

\subsubsection{Proof of Theorem~\ref{Thm:OddEquipartingMatrices}}
The $k$-bit Gray codes $G$ with transition count $c_k(G) = 1$ are exactly given by composition of two $(k-1)$-bit Gray codes.
Let $\mathfrak{c}_1 + \dots + \mathfrak{c}_{k-1} = 2(2^{k-1} -1)$ be such that the number of non-equivalent $(\mathfrak{c}_1,\dots,\mathfrak{c}_{k-1})$-equiparting matrices is odd.
Then it holds that $\mathfrak{c}_1 + \dots + \mathfrak{c}_{k-1} + 1 = (2^k - 1)$ and the number of non-equivalent $(\mathfrak{c}_1,\dots,\mathfrak{c}_{k-1},1)$-equiparting matrices is odd,
which implies $E(1,k) \leq E(2, k-1)$.
$I(1,k) = 2^{k-1}$ is a consequence of Theorem~\ref{Thm:ShuffleCount}.

Thus it suffices to verify the values of $E(j,k)$ and $I(j,k)$ for $j > 1$.
In order to do so, we use delta sequences and make a few observations:

\begin{definition}
    A $j$-tuple of elements in $[k]^{2^k -1}$
    \[
        \mathcal{D} = (D_1,\dots,D_j)
    \]
    is a $(\mathfrak{c}_1,\dots,\mathfrak{c}_k)$-delta sequence, if it is a delta sequence of a $(\mathfrak{c}_1,\dots,\mathfrak{c}_k)$-equiparting matrix
    \[
        \mathcal{G} = (G_1,\dots,G_j),
    \]
    i.e. for each $i = 1,\dots, j$  it holds that $\delta(G_i) = D_i$.
\end{definition}
The number of non-isomorphic $(\mathfrak{c}_1,\dots,\mathfrak{c}_k)$-equiparting matrices is equal to the number of $(\mathfrak{c}_1,\dots,\mathfrak{c}_k)$-delta sequences.

As before $(\mathbb{Z}/2)^k$ acts trivially on $(\mathfrak{c}_1,\dots,\mathfrak{c}_k)$-delta sequences and as above $\stabsym$ acts freely by diagonal action.
We say that two $(\mathfrak{c}_1,\dots,\mathfrak{c}_k)$-delta sequences are equivalent, if they are in one $\stabsym$-orbit.
Thus the number of non-equivalent $(\mathfrak{c}_1,\dots,\mathfrak{c}_k)$-equiparting matrices is equal to the number of non-equivalent $(\mathfrak{c}_1,\dots,\mathfrak{c}_k)$-delta sequences.

We denote by $N(\mathfrak{c}_1, \dots, \mathfrak{c}_k)$ the number of non-equivalent $(\mathfrak{c}_1,\dots,\mathfrak{c}_k)$-delta sequences modulo $2$.
Understanding this for all $\mathfrak{c}_1,\dots,\mathfrak{c}_k$ determines $E(j,k)$.
As $N(\mathfrak{c}_1,\dots,\mathfrak{c}_k)$ is invariant under the action of $\sym_k$,
we may ass well assume that $\mathfrak{c}_1 \geq \mathfrak{c}_2 \geq \dots \geq \mathfrak{c}_k$.

Understanding $N(\mathfrak{c}_1,\dots,\mathfrak{c}_k)$ for all $\mathfrak{c}_1 > \dots > \mathfrak{c}_k$ will determine $I(j,k)$:
\begin{lemma}\label{Lem:ExactEquivalent}
    If the number of non-isomorphic $(\mathfrak{c}_1,\dots,\mathfrak{c}_k$)-equiparting is odd, then the $\mathfrak{c}_1,\dots,\mathfrak{c}_k$ are pairwise distinct.

    If the $\mathfrak{c}_1,\dots,\mathfrak{c}_k$ are pairwise distinct, then
    $(\mathfrak{c}_1,\dots,\mathfrak{c}_k)$-equiparting matrices
    are isomorphic if and only if the are equivalent.
\end{lemma}
\begin{proof}
    As $\stabsym^\pm$ acts freely on the set of $(\mathfrak{c}_1,\dots,\mathfrak{c}_k)$-equiparting matrices,
    the number of $(\mathfrak{c}_1,\dots,\mathfrak{c}_k)$-equiparting matrices is divisible by $|\stabsym^\pm| = |(\Z/2)^k| \cdot |\stabsym|$.
    Hence the number of non-isomorphic $(\mathfrak{c}_1,\dots,\mathfrak{c}_k)$-equiparting matrices is even unless $|\stabsym|$ is odd,
    which is equivalent to the $\mathfrak{c}_1,\dots,\mathfrak{c}_k$ being pairwise distinct.

    The second part is clear.
    In this case, $\stabsym$ is trivial.
\end{proof}

To determine this parity it suffices to consider only the standard $k$-bit Gray code:
\begin{lemma}\label{Lem:PermuteOneDeltaSequence}
    Let $j,k \geq 2$ and let $\mathfrak{c}_1 + \dots + \mathfrak{c}_k = j(2^k-1)$.
    \begin{enumerate}
        \item
            Let $(D_1,\dots,D_j)$ be a $(\mathfrak{c}_1,\dots,\mathfrak{c}_k)$-delta sequence.
            Let $1 \leq i \leq j$ and let $\widetilde{D}$ be a delta sequence with $c(\widetilde{D}) = c(D_i)$.
            Then
            \[
                (D_1,\dots,D_{i-1},\widetilde{D},D_{i+1},\dots,D_j)
            \]
            is equivalent to $(D_1,\dots,D_j)$ if and only if $D_i = \widetilde{D}$.
        \item
            The parity $N(\mathfrak{c}_1,\dots,\mathfrak{c}_k)$ is determined by considering only permutations of delta sequences of the standard $k$-bit Gray code.
    \end{enumerate}
\end{lemma}
\begin{proof}
    \begin{enumerate}
        \item
            Let $\tau \in \stabsym$
            with
            \[
                \tau \cdot (D_1,\dots, D_j) = (D_1,\dots,D_{i-1},\widetilde{D},D_{i+1},\dots,D_j).
            \]
            The group $\sym_k$ acts freely on delta sequences and as $j > 1$ it follows that $\tau = \id$ and $\widetilde{D} = D_i$.
        \item
            Follows from (1) and Theorem~\ref{Thm:ShuffleCount}.
    \end{enumerate}
\end{proof}

We observe another group action:
$\sym_j$ acts on $(\mathfrak{c}_1,\dots,\mathfrak{c}_k)$-delta sequences by
\[
    \sigma \cdot (D_1,\dots,D_j) := (D_{\sigma^{-1}(1)}, \dots, D_{\sigma^{-1}(j)}).
\]
This action commutes with the action of $\stabsym$ and it
is the key in determining $N(\mathfrak{c}_1,\dots,\mathfrak{c}_k)$:
If the $\sym_j$-orbit of $(D_1,\dots,D_j)$ contains an even number of non-equivalent elements, then $(D_1,\dots,D_j)$ need not be considered for the parity of $N(\mathfrak{c}_1,\dots,\mathfrak{c}_k)$.

Suppose that $N(\mathfrak{c}_1,\dots,\mathfrak{c}_k)$ is odd.
Let $j = 2^n + r$.
In the next Proposition we will see that we only need to consider $(D_1,\dots,D_j)$, where $2^n$ of them are identical.
We have already seen that we only need to consider permutations of the standard $k$-bit Gray code, which has transition counts $(2^{k-1},\dots,2,1)$.
It follows then that $\mathfrak{c}_1 \geq 2^n2^{k-1} + r\cdot 1$.
Before we proceed, recall some number theory and introduce a term:

The dyadic valuation of an integer $n$ is the largest integer $m$ such that $2^m$ divides $n$.
To determine whether a fraction is odd, one can compare the dyadic valuation of numerator and denominator.
\begin{lemma}\label{Lem:Legendre}
    The dyadic valuation of $\binom{j}{a}$ is
    \[
        \nu_2 \left( \binom{j}{a} \right) = s_2(a) + s_2(j - a) - s_2(j),
    \]
    where $s_2(d)$ is the sum of the dyadic digits of $j$.
\end{lemma}
\begin{proof}
    By a formula of Legendre, see \cite[Theorem~2.6.4]{Moll},
    the dyadic valuation of $j!$ is given by
    \[
        \nu_2(j!) = j - s_2(j).
    \]
\end{proof}
\begin{lemma}\label{Lem:Sub}
    The sum of the dyadic digits of positive numbers is subadditive and submultiplicative:
    Let $n, m \geq 1$ be integers.
    \begin{enumerate}
        \item
            \[
                s_2(n) + s_2(m) \geq s_2(n + m)
            \]
            and the inequality is tight, exactly if the dyadic digits of $n$ and $m$ decompose the dyadic digits of $n + m$.
        \item
            \[
                s_2(n)s_2(m) \geq s_2(nm).
            \]
    \end{enumerate}
\end{lemma}
\begin{proof}
    We assume the dyadic digits of $n$ and $m$ to be given as
    \[
        n = \sum_{i=0}^r 2^i n_i, \quad m = \sum_{i=0}^r 2^i m_i.
    \]
    \begin{enumerate}
        \item
            That the inequality is tight, when the dyadic digits are decomposed is clear.

            Now by induction on the number of digits of $m$, it suffices to show that for an integer $j \geq 1$ with $n_j = 1$ it holds that
            \[
                s_2(n) + 1 > s_2(n + 2^j).
            \]
            We induce for fixed $n$ on $j$ from above (note that $j$ is bounded by $n$).
            It holds that $s_2(n) = s_2(n - 2^j) + 1$.
            Then
            \[
                s_2(n + 2^j) = s_2(n - 2^j + 2^{j+1}) \leq s_2(n - 2^j) + s_2(2^{j+1}) = s_2(n) + 1 + 1.
            \]
            If $n_{j+1} = 0$, the last inequality holds tight.
            If $n_{j+1} = 1$, the last inequality holds by induction hypothesis.
            The later is never the case for the base case.
        \item
            With
            \[
                nm = \sum_{k = 0}^{2r} \sum_{i + j = k} 2^i 2^j m_i n_i
            \]
            the statment follows from (1).
    \end{enumerate}
\end{proof}

\begin{lemma}\label{Lem:IneqNumb}
    Let $T \geq 1$ be an integer.
    It holds that
    \[
        T - \nu_2(T) \geq s_2(T)
    \]
    and this inequality is tight exactly for $T \in \{1, 2\}$.
\end{lemma}
\begin{proof}
    It is clear that the inequality is tight for $T \in \{1, 2\}$.
    If $T$ is odd and at least $3$ then $T - \nu_2(T) = T > s_2(T)$.
    The remaining cases follow by induction.
    If $T - \nu_2(T) \geq s_2(T)$ for some $T \geq 2$, then
    \begin{align*}
        2T - \nu_2(2T) &= 2T - \nu_2(T) - 1\\
                       &\geq s_2(T) + T - 1\\
                       &= s_2(2T) + T - 1 > s_2(2T).
    \end{align*}
\end{proof}

\begin{definition}
    The \defn{multiplicity} of $(D_1,\dots,D_j)$ is the multiset formed by:
    \[
        | \{i \in \{1,\dots,j\} \colon D_i = D \} |
    \]
    for all $D \in \{D_1,\dots,D_j\}$.
\end{definition}
Note that the multiplicity of a $(\mathfrak{c}_1,\dots,\mathfrak{c}_k)$-delta sequence is $\stabsym$-invariant and $\sym_j$-invariant.
\begin{example}
    Let $D, D', D''$ be pairwise distinct.
    \begin{itemize}
        \item $(D, D, D', D')$ has multiplicity $(2, 2)$.
        \item $(D, D, D, D', D')$ has multiplicity $(2, 3)$.
        \item $(D, D, D', D'', D'')$ has multiplicity $(1, 2, 2)$.
    \end{itemize}
\end{example}

The multiplicites of those delta sequences we need to consider are limited by $j$:

\begin{proposition}\label{Prop:Multiplicity}
    The parity $N(\mathfrak{c}_1,\dots,\mathfrak{c}_k)$ is determined by considering only $(\mathfrak{c}_1,\dots,\mathfrak{c}_k)$-delta sequences of multiplicity
    \begin{enumerate}
        \item $(a_1,\dots,a_\ell)$, where the dyadic digits of the $a_i$ decompose the dyadic digits of $j$.
        \item $(\frac{j}{2}, \frac{j}{2})$, where $j$ is a power of two and $|\stabsym| \geq 2$.
    \end{enumerate}
\end{proposition}
\begin{proof}
    Let $\mathcal{D} := (D_1,\dots,D_j)$ be a $(\mathfrak{c}_1,\dots,\mathfrak{c}_k)$-delta sequence of multiplicity $(a_1,\dots,a_\ell)$.
    As $\sym_j$ and $\stabsym$ commute, the number of non-equivalent elements of the $\sym_j$-orbit of $\mathcal{D}$ is equal to the size of the $\sym_j$-orbit divided by the cardinality $T$ of
    \[
        \mathcal{T} := \left \{ \tau \in \stabsym \mid \exists \sigma \in \sym_j \colon \tau \mathcal{D} = \sigma \mathcal{D} \right \}.
    \]

    If $\tau \in \mathcal{T}$ and $\tau(D_{i_1}) = D_{i_2}$, both $D_{i_1}$ and $D_{i_2}$ must have the same cardinality in $(D_1,\dots,D_j)$.
    As $\stabsym$ acts freely on delta sequences,
    $\tau(D_{i_1}) = D_{i_2}$ completely determines $\tau$
    and the $(a_1,\dots,a_\ell)$ can be sorted as follows:
    \[
        (
            a_1,\dots,a_{\frac{\ell}{T}}, \quad
            a_1,\dots,a_{\frac{\ell}{T}}, \quad
            \dots, \quad
            a_1,\dots,a_{\frac{\ell}{T}}
        ).
    \]

    Note that $T(a_1 + \dots + a_\ell) = j$ and $j$ must be divisible by $T$.
    There are
    \[
        S:= \binom{j}{a_1}\binom{j-a_1}{a_2} \dots \binom{j-a_1-a_2-\dots-a_{\ell-1}}{a_\ell}
    \]
    elements in the $\sym_j$-orbit of $\mathcal{D}$.
    Hence
    \begin{align*}
        \nu_2(S)
        &= \sum_{i=1}^\ell \Big(s_2(a_i) + s_2(j - a_1 - \dots - a_i) - s_2(j - a_1 - \dots - a_{i-1})\Big)\\
        &= s_2(0) - s_2(j) + \sum_{i=1}^\ell s_2(a_i).
    \end{align*}
    If the number of non-equivalent elements of the $\sym_j$-orbit of $\mathcal{D}$ is even, $\mathcal{D}$ may be disregarded for the parity of $N(\mathfrak{c}_1,\dots,\mathfrak{c}_k)$.
    So suppose that it is odd, it follows that~$\nu_2(T) = \nu_2(S)$. However,
    \begin{align}
        \nu_2(S)
        \nonumber
        &= - s_2(j) + T \sum_{i=1}^{\frac{\ell}{T}} s_2\left(a_{i}\right)\\
        \nonumber
        &\geq - s_2(T) s_2\left(\frac{j}{T}\right) + T \sum_{i=1}^{\frac{\ell}{T}} s_2\left(a_{i}\right)\\
        \label{Eq:1}\tag{1}
        &\stackrel{(\ref{Eq:1})}{\geq} - \left(T - \nu_2(T)\right)s_2\left(\frac{j}{T}\right) + T \sum_{i=1}^{\frac{\ell}{T}} s_2\left(a_{i}\right)\\
        \nonumber
        &= \nu_2(T)s_2\left(\frac{j}{T}\right) + T \left(- s_2\left(\frac{j}{T}\right) + \sum_{i=1}^{\frac{\ell}{T}} s_2\left(a_{i}\right)\right)\\
        \label{Eq:2}\tag{2}\nonumber
        &\stackrel{(\ref{Eq:2})}{\geq} \nu_2(T)s_2\left(\frac{j}{T}\right)\\
        \label{Eq:3}\tag{3}\nonumberthis
        &\stackrel{(\ref{Eq:3})}{\ge} \nu_2(T).
    \end{align}
    Note, that~\eqref{Eq:1} follows from Lemma~\ref{Lem:Sub} and that~\eqref{Eq:2} follows from Lemma~\ref{Lem:IneqNumb}.
    So the above inequalities must be tight.
    This implies:
    \begin{itemize}
        \item $T \in \{1, 2\}$ by~\eqref{Eq:1} and Lemma~\ref{Lem:Sub}.
        \item The dyadic digits of $a_{1}, \dots, a_{{\frac{\ell}{T}}}$ decompose the dyadic digits of $\frac{j}{T}$ by~\eqref{Eq:2} and Lemma~\ref{Lem:IneqNumb}.
        \item $T$ is odd or $\frac{j}{T}$ is a power of two by~\eqref{Eq:3}.
    \end{itemize}
\end{proof}

To determine $N(\mathfrak{c}_1,\dots,\mathfrak{c}_k)$ it now suffices to consider those $(\mathfrak{c}_1,\dots,\mathfrak{c}_k)$-delta sequences,
such that
\begin{enumerate}
    \item each $D_i$ is the delta sequence of a permutation of a standard $k$-bit Gray code by Lemma~\ref{Lem:PermuteOneDeltaSequence} and
    \item the multiplicity $(a_1,\dots,a_\ell)$ of $(D_1,\dots,D_j)$ satisfies the conditions of Propostion~\ref{Prop:Multiplicity}.
\end{enumerate}

Suppose that $(D_1,\dots,D_j)$ is a sequence that satisfies both.
Recall that the standard $k$-bit Gray code has transition counts $(2^{k-1},2^{k-2},\dots,1)$.
Hence,
\[
    \max(\mathfrak{c}_1,\dots,\mathfrak{c}_k) = a_1 \cdot 2^{k-1} + (j - a_1) \cdot 1,
\]
as $a_1$ of the $D_1,\dots,D_j$ have transition count $2^{k-1}$ for some row.

If $j = 2^n$ and the $\mathfrak{c}_1,\dots,\mathfrak{c}_k$ are not pairwise distinct, we have that $(a_1,\dots,a_\ell)$ is one of $(2^n)$ and $(2^{n-1}, 2^{n-1})$.
In this case
\[
    \max(\mathfrak{c}_1,\dots,\mathfrak{c}_k) \geq 2^n 2^{k-1} + 2^n.
\]
Otherwise, if $j = 2^n + r$ for $0 \leq r < 2^n$, we have that $a_1$ is at least $2^n$ and we conclude
\[
    \max(\mathfrak{c}_1,\dots,\mathfrak{c}_k) \geq 2^n 2^{k-1} + j.
\]
The first case implies that
\[
    E(2^{n-1} + 2^{n-1}, k) \geq 2^{n-1}2^{k-1} + 2^{n-1}
\]
and the second case implies that
\[
    E(2^{n} + r, k) \geq 2^{n}2^{k-1} + r \leq I(2^n + r, k).
\]

This bounds the values of $E(j,k)$ and $I(j,k)$ from below according to Theorem~\ref{Thm:OddEquipartingMatrices}.
To bound the value from above and conclude our proof,
we need to find some values for which $N(\mathfrak{c}_1,\dots,\mathfrak{c}_k)$ is odd.

The case $k = 2$ needs special attention, but was also known before:
\begin{proposition}[{See \cite[Section~4.3.1]{BFHZ2016}}]
    Let $k = 2$ and let $2 \leq j = 2^n + r$ with $0 \leq r \leq 2^n$.
    Then
    \[
        N(2^n + 2r, 2^n2 + r) = 1.
    \]
\end{proposition}
Note that $2^n + 2r \neq 2^n2 + r$ unless $r = 2^n$.
\begin{proof}
    There exist exactly two $2$-bit delta sequences with transition counts $(1,2)$ resp.~$(2,1)$.
    Hence the number of $(2^n + 2r, 2^n 2 + r)$-delta sequences is $\binom{j}{r} = \binom{2^n + r}{r}$.
    By Lemma~\ref{Lem:Legendre}:
    \[
        \nu_2\left(\binom{2^n + r}{r}\right) = s_2(r) + s_2(2^n) - s_2(2^n + r) = s_2(r) + 1 - s_2(2^n + r) = \begin{cases} 0, & \text{if } 0 \leq r < 2^n,\\
        1, & \text{if } r = 2^n. \end{cases}
    \]
    We see that if $r \neq 2^n$, then the number of $(2^n + 2^r, 2^n 2 + r)$-delta sequences is odd and $N(2^n + 2r, 2^n 2 + r)$ must be odd as well.
    If $r = 2^n$, then the number of $(2^n + 2^r, 2^n 2 + r)$-delta sequences is even but not divisible by $4$.
    However, $|\sym_{(2^n + 2r, 2^n2 + r)}| = 2$, so the delta sequences come in equivalent pairs and $N(2^n + 2r + 2^n 2 + r)$ is also odd in this case.
\end{proof}
The next step is also intuitive.
Combining $2^n$ delta sequences with transition counts $(2,4,\dots,2^{k-1},1)$ and $2^n$ delta sequences with transition counts $(2,4,\dots,2^{k-2},1,2^{k-1})$ will give the only case of
\[
    2^n \cdot (2 + 2, 4 + 4, \dots, 2^{k-2} + 2^{k-2}, 2^{k-1} + 1, 2^{k-1} + 1)
\]-equiparting matrices we need to consider and will show that
$E(2^{n+1}, k) = 2^{n+k-1} + 2^n$:
\begin{proposition}\label{Prop:ExOddCount}
    Let $k \geq 2$ and $j = 2^{n+1}$
    Then
    \[
        N(2^{n+2},2^{n+3},\dots,2^{n+k-1}, 2^{n + k-1} + 2^n, 2^{n + k-1} + 2^n)
    \]
    is odd.
\end{proposition}
\begin{proof}
    Note that
    \[
        |\sym_{(2^{n+2},2^{n+3},\dots,2^{n+k-1}, 2^{n + k-1} + 2^n, 2^{n + k-1} + 2^n)}| = 2.
    \]
    We only need to consider those delta sequences of permutations of standard $k$-bit Gray codes of multiplicity $(2^{n+1})$ or $(2^n, 2^n)$.
    Such sequences do not exist for multiplicity $(2^{n+1})$ so we consider the other case.

    Suppose $\mathcal{D} = (D_1,\dots,D_j)$ is such a sequence.
    As noted above, the transition count $c_i(\mathcal{D}) \geq 2^{k-1} 2^n  + 2^n$ for two of the $i$.
    It immediatly follows that $2^n$ of the $(D_1,\dots,D_j)$ have transition counts
    $(2, 4, \dots, 2^{k-2}, 1, 2^{k-1})$
    and the others have transition counts
    $(2, 4, \dots, 2^{k-2}, 2^{k-1}, 1)$.

    The $\sym_j$-orbit has size $\binom{2^{n+1}}{2^n}$, which is divisible by two, but not by four.
    Those come in pairs of equivalent permutations, hence the number of non-equivalent elements is odd.
\end{proof}

The other remaining cases are a bit more difficult to see.
With the help of a computer we see some examples:

\begin{examples}
    Let $k = 4$.
    The following are all $\mathfrak{c}_1,\dots,\mathfrak{c}_4$ for different $j$ with
    $N(\mathfrak{c}_1,\dots,\mathfrak{c}_4)$ odd and
    \[
        \mathfrak{c}_1  < \mathfrak{c}_2 < \mathfrak{c}_3 < \mathfrak{c}_4
    \]
    such that $\mathfrak{c}_4$ is minimal with this property:
    \begin{enumerate}[$j=1$:]
        \item $(1, 2, 4, 8)$,
        \item $(2, 4, 8, 16)$,
        \item $(4, 8, 16, 17)$,
        \item $(4, 8, 16, 32)$,
        \item $(6, 12, 24, 33)$, $(6, 16, 20, 33)$, $(8, 10, 24, 33)$, $(10, 12, 20, 33)$, $(8, 16, 18, 33)$,
        \item $(8, 16, 32, 34)$,
        \item $(12, 24, 34, 35)$, $(16, 20, 34, 35)$,
        \item $(8, 16, 32, 64)$,
        \item $(10, 20, 40, 65)$,
            $(10, 24, 36, 65)$,
            $(12, 18, 40, 65)$,
            $(12, 24, 34, 65)$,
            $(16, 18, 36, 65)$,
            $(16, 20, 34, 65)$,
        \item
            $(12, 24, 48, 66)$,
            $(12, 32, 40, 66)$,
            $(16, 20, 48, 66)$,
            $(16, 32, 36, 66)$,
            $(20, 24, 40, 66)$,
        \item
            $(14, 28, 56, 67)$,
            $(14, 32, 52, 67)$,
            $(14, 36, 48, 67)$,
            $(14, 40, 44, 67)$,
            $(16, 26, 56, 67)$,
            $(16, 32, 50, 67)$,\newline
            $(16, 34, 48, 67)$,
            $(16, 40, 42, 67)$,
            $(18, 24, 56, 67)$,
            $(18, 28, 52, 67)$,
            $(18, 36, 44, 67)$,
            $(20, 22, 56, 67)$,\newline
            $(20, 26, 52, 67)$,
            $(20, 36, 42, 67)$,
            $(20, 38, 40, 67)$,
            $(22, 24, 52, 67)$,
            $(22, 28, 48, 67)$,
            $(22, 32, 44, 67)$,\newline
            $(24, 26, 48, 67)$,
            $(24, 32, 42, 67)$,
            $(24, 34, 40, 67)$,
            $(24, 36, 38, 67)$,
            $(26, 28, 44, 67)$,
        \item $(16, 32, 64, 68)$,
        \item $(18, 40, 68, 69)$, $(20, 40, 66, 69)$, $(24, 34, 68, 69)$, $(24, 36, 66, 69)$,
        \item $(24, 48, 68, 70)$, $(32, 40, 68, 70)$
        \item $(28, 56, 70, 71)$, $(32, 52, 70, 71)$, $(36, 48, 70, 71)$, $(40, 44, 70, 71)$
    \end{enumerate}
\end{examples}

Apparently there are many values, to verify that the lower bound is achieved.
One observation is that there is always a tuple with $2 \mathfrak{c}_1 = \mathfrak{c}_2$.
This generalizes:
\begin{examples}
    Let $k = 5$.
    The following are all $\mathfrak{c}_1,\dots,\mathfrak{c}_5$ for different $j$ with
    $N(\mathfrak{c}_1,\dots,\mathfrak{c}_5)$ odd and
    \[
        4\mathfrak{c}_1  = 2\mathfrak{c}_2 = \mathfrak{c}_3 < \mathfrak{c}_4 < \mathfrak{c}_5
    \]
    such that $\mathfrak{c}_5$ is minimal with this property:
    \begin{enumerate}[$j=1$:]
        \item $(1,2,4,8,16)$,
        \item $(2, 4, 8, 16, 32)$,
        \item $(4, 8, 16, 32, 33)$,
        \item $(4, 8, 16, 32, 64)$,
        \item $(6, 12, 24, 48, 65)$, $(8, 16, 32, 34, 65)$,
        \item $(8, 16, 32, 64, 66)$
        \item $(12, 24, 48, 66, 67)$
        \item $(8, 16, 32, 64, 128)$
        \item $(10, 20, 40, 80, 129)$, $(12, 24, 48, 66, 129)$
        \item $(12, 24, 48, 96, 130)$, $(16, 32, 64, 68, 130)$
        \item $(14, 28, 56, 112, 131)$, $(16, 32, 64, 98, 131)$
        \item $(16, 32, 64, 128, 132)$
        \item $(20, 40, 80, 130, 133)$
        \item $(24, 48, 96, 132, 134)$
    \end{enumerate}
\end{examples}

Apparently we may start with $2^n$ times transition counts $(1, 2, 4, 8, 16)$,
fill up with $(2, 4, 8, 16, 1)$ until we reach $j = 2^n + 2^{n-1}$ and the remaining onces
are $(4, 8, 16, 2, 1)$:

\begin{proposition}
    Let $k > 2$ and $2 \leq j = 2^n + r$ with $0 \leq r < 2^n$.
    Define integers $w = \min(r, 2^{n-1})$ and $s = r - w$ such that $j = 2^n + w + s$.
    Then
    \[
        N\left(\mathfrak{c}_1,\dots,\mathfrak{c}_k\right) = 1,
    \]
    where
    \[
        (\mathfrak{c}_1,\dots,\mathfrak{c}_k) =
        2^n \cdot (1, 2, 4, \dots, 2^{k-1}) + w \cdot (2, 4, \dots, 2^{k-1}, 1) + s \cdot (4, 8, \dots, 2^{k-1}, 2, 1)
    \]
\end{proposition}
\begin{proof}
    It holds that
    \[
        \mathfrak{c}_{k-1} - \mathfrak{c}_{k-2} = 2^{k-3}2^n + 2^{k-2}w - (2^{k-1} - 2)s >
        \begin{cases}
            0, & \text{if } s = 0, \\
            2^{k-3}2w + 2^{k-2}w - 2^{k-1}w = 0, & \text{otherwise},
        \end{cases}
    \]
    and
    \[
        \mathfrak{c}_{k} - \mathfrak{c}_{k-1} = 2^{k-2}2^n + w + s - 2^{k-1}w - 2s.
    \]
    Hence
    \[
        \mathfrak{c}_1 < \mathfrak{c}_2 < \dots < \mathfrak{c}_{k-2} < \mathfrak{c}_{k-1} < \mathfrak{c}_k.
    \]
    as $r < 2^n$.
    This implies that
    \[
        |\stabsym| = 1.
    \]
    By Proposition~\ref{Prop:Multiplicity} it suffices to consider
    $(\mathfrak{c}_1,\dots,\mathfrak{c}_k)$-delta sequences of multiplicity $(a_1,a_2,\dots,a_\ell)$,
    where we may assume that $a_1 \geq 2^n$.
    If $w = 2^{n-1}$, then we may assume that $a_2 \geq 2^{n-1}$ or that $a_1 \geq 2^n + 2^{n-1}$.
    By Lemma~\ref{Lem:PermuteOneDeltaSequence} it suffices to consider those sequences that consist
    of permutations of the standard $k$-bit Gray code.

    Suppose $\mathcal{D} = (D_1,\dots,D_j)$ is such sequence.
    W.l.o.g.\ we reorder $\mathcal{D}$ such that $D_1 = \dots = D_{2^n}$
    and such that $a$ is maximal with $D_{2^n + 1} = \dots = D_{2^n + a}$.
    We will show that $(a_1,\dots,a_\ell) = (2^n, w, s)$ and that $a = w$ and that
    \[
        c(D_m) = \begin{cases}
            \left(1,2,4,\dots,2^{k-1}\right), & \text{if } m \leq 2^n, \\
            \left(2,4,\dots,2^{k-1}, 1\right), & \text{if } 2^n < m \leq 2^n + w, \\
            \left(4,8,\dots,2^{k-1}, 2, 1\right), & \text{otherwise}.
        \end{cases}
    \]

    \smallskip

    As above $2^n$ of the $(D_1,\dots,D_j)$ are equal.
    Hence at least one of the transition counts is $2^{k-1}2^n + r = 2^{k-1}2^n + w + s$.
    As $\mathfrak{c}_k$ has this value and is strictly the largest, we obtain
    \[
        c_k(D_m) = \begin{cases}
            2^{k-1}, & \text{if }m \leq 2^n,\\
            1, & \text{otherwise}.
        \end{cases}
    \]

    Suppose $s > 0$. This implies $w = 2^{n-1}$ and
    \[
        \mathfrak{c}_{k-1} \equiv 2s \mod 2^{n}
    \]
    with $2s < 2^{n}$.
    However,
    \[
        \nu_2\left(c_{k-1}(D_1)\right) = \dots = \nu_2\left(c_{k-1}(D_{2^n})\right) \geq 0
    \]
    and
    \[
        \nu_2\left(c_{k-1}(D_{2^n + 1})\right) = \dots = \nu_2\left(c_{k-1}(D_{2^n + w})\right) \geq 1.
    \]
    Therefore,
    \[
        c_{k-1}(D_1) + \dots + c_{k-1}(D_{j}) \equiv c_{k-1}(D_{2^n + w + 1}) + \dots + c_{k-1}(D_{j}) \mod 2^n.
    \]
    As they are all at least $2$, we conclude that
    \[
        c_{k-1}(D_{2^n + w + 1}) = \dots = c_{k-1}(D_{j}) = 2.
    \]

    Next we use induction on $i$ to determine the $c_i(D_m)$ for $i \leq k-2$.
    By induction we have that
    \[
        c_i(D_m) \geq \begin{cases}
            2^{i-1}, & \text{if } m \leq 2^n, \\
            2^{i}, & \text{if } 2^n < m \leq 2^n + w,\\
            2^{i+1}, & \text{if } 2^n + w < m \leq 2^n + w + s.
        \end{cases}
    \]
    As $\mathfrak{c}_i = 2^{i-1}2^n + 2^iw + 2^{i+1}s$, it follows that the inequalities are tight.
    We conclude the claimed values for $c_1(D_m),\dots,c_{k-2}(D_m)$ and $c_{k}(D_m)$ for all $m$.
    The values for $c_{k-1}(D_m)$ follow as they are the only values left.

    \smallskip

    The $\sym_j$-orbit of $\mathcal{D}$ has size $\binom{2^n + w + s}{s}\binom{2^n + w}{w}$, which is not divisible by two.
\end{proof}

\section{Reduction to Equiparting matrices}\label{Sec:Reduction}

We will briefly summarize the approach in \cite{BFHZ2016} along with a slight generalization.
Via the inclusion
\[
    \iota \colon \mathbb{R}^d \to \mathbb{R}^{d+1}, \quad (y_1,\dots,y_d) \mapsto (1,y_1,\dots,y_d)
\]
we parametrize oriented affine hyperplanes in $\mathbb{R}^d$ by $S^d$
with two non-proper hyperplanes that cannot correspond to solutions.
A vector $v \in S^d$ corresponds to the oriented hyperplane
\[
    \left\{r \in \mathbb{R}^d \colon \langle \iota(r), v \rangle = 0 \right\}
\]
with induced positive and negative side.

A collection $\mathcal{M} = (\mu_1,\dots, \mu_j)$ of finite Borel measures on $\mathbb{R}^d$ induces
a map
\begin{align*}
    \psi_{\mathcal{M}} \colon \quad X_{d,k} = (S^d)^{*k} \quad &\to \quad W_k \oplus (U_k)^{\oplus j} \cong \mathbb{R}^{k-1} \oplus \big(\mathbb{R}^{(\mathbb{Z}/2)^k}\big)^{\oplus j}\\
    \lambda_1 v_1 + \dots + \lambda_k v_k \quad &\mapsto \quad (\lambda_1 - \frac{1}{k}, \dots, \lambda_k - \frac{1}{k}) \oplus (\lambda_1 \cdots \lambda_k) \cdot \phi_\mathcal{M}(v_1,\dots, v_k),
\end{align*}
where $\phi_\mathcal{M}$ assigns $k$ oriented hyperplanes with normals $(v_1,\dots,v_k)$ to each measure evaluated on each of the (possibly empty) $2^k$ regions minus $\frac{1}{2^k}\mu_i(\mathbb{R}^d)$.
This map is $\sym_k^{\pm}$-equivariant as explained in \cite[Section~2.3]{BFHZ2016}.
If $0$ is in the image for all collections of measures, then $\Delta(j,k) \leq d$.

If $0$ lies not in the image, then $\psi_\mathcal{M}$ can be composed with the radial projection
$\nu \colon (W_k \oplus U_k^{\oplus j}) \setminus \{0\} \to S(W_k \oplus U_k^{\oplus j})$.

Denote by $X_{d,k}^{>1}$ the subset of those points in $X_{d,k}$ of non-trivial stabilizer.
The image of $X_{d,k}^{>1}$ does not contain $0$ and in fact any two maps $\psi_\mathcal{M}$ and $\psi_{\mathcal{M}'}$ restricted to $X_{d,k}^{>1}$ are $\sym_k^{\pm}$-homotopic~\cite[Prop~2.2]{BFHZ2016}. This yields:

\begin{proposition}[{\cite[Thm.~2.3~(ii)]{BFHZ2016}}]\label{Prop:FirstTop}
    Let $d, k,j  \geq 1$ be integers and let $\mathcal{M} = (\mu_1, \dots, \mu_j)$ be finite Borel measures on $\mathbb{R}^d$.

    If there is no $\sym_{k}^{\pm}$-equivariant map
    \[
        X_{d,k} \to S(W_k \oplus U_k^{\oplus j})
    \]
    whose restriction to $X_{d,k}^{>1}$ is $\sym_k^\pm$-homotopic to $\nu \circ \psi_\mathcal{M} \mid_{X_{d,k}^{>1}}$, then $\Delta(j,k) \leq d$.
\end{proposition}
\begin{corollary}\label{Cor:FirstTopExact}
    If there is no $(\Z/2)^k$-equivariant map
    \[
        X_{d,k} \to S(W_k \oplus U_k^{\oplus j})
    \]
    whose restriction to $X_{d,k}^{>1}$ is $(\Z/2)^k$-homotopic to $\nu \circ \psi_\mathcal{M} \mid_{X_{d,k}^{>1}}$, then $\Delta(j,k) \leq d$.
\end{corollary}

There are different approaches, how to show the non-existent of such equivariant maps.
We will use obstruction theory as developed in \cite{BFHZ2016}.
As a first step, one can equip $X_{d,k}$ with the CW-structure developed in~\cite[Section~3]{BFHZ2016}.
We assume the reader to be familiar with this CW-structure.
We remark that $X_{d,k}^{>1}$ is a subcomplex, which allows us to use relative equivariant obstruction theory.
The definition of a (relatively open) cell is as follows:
\begin{definition}[{See~\cite[Sections~2.2 and~2.3]{BFHZ2016}}]\label{def:Cell}
    Let $(\sigma_1,\dots,\sigma_k)$ be a permutation of $1,\dots,k$.
    Let $(s_1,\dots,s_k) \in \{+1, -1\}^k$ and let $(i_1,\dots,i_k) \in \{1,\dots, d+2\}^k$.
    Then
    \begin{align*}
        C_{i_1,\dots,i_k}^{s_1,\dots,s_k}(\sigma_1,\dots,\sigma_k) :=
        \left\{
            (x_1,\dots,x_k) \in \mathbb{R}^{(d+1) \times k} \colon 0 <_{i_1} s_1x_{\sigma_1} <_{i_2} \dots <_{i_k} s_k x_{\sigma_k}
        \right\},
    \end{align*}
    where $y <_i y'$ means that $y$ and $y'$ agree in the first $i-1$ coordinates and $y_i < y_i'$.

    This induces a (relatively open) cell
    \[
        D_{i_1,\dots,i_k}^{s_1,\dots,s_k}(\sigma_1,\dots,\sigma_k) :=
        C_{i_1,\dots,i_k}^{s_1,\dots,s_k}(\sigma_1,\dots,\sigma_k)
        \cap S(\mathbb{R}^{(d+1) \times k}).
    \]
\end{definition}

We proceed as described in \cite[Section~2.6]{BFHZ2016}.
Let $N_2 = (2^k - 1)j + k - 2$ be the dimension of the sphere $S(W_k \oplus U_k^{\oplus j})$ and let $\theta$ be some $(N_2 + 1)$-cell of $X_{d,k}$ and let $Z$ be the union of $X_{d,k}^{>1}$ with the $\sym_k^\pm$-orbit of the relative closure of $\theta$.
It suffices to show that the map $\nu \circ \psi_\mathcal{M} \mid_{X_{d,k}^{>1}}$ cannot be $\sym_k^\pm$-equivariantly extended to $Z$.
We may also use Corollary~\ref{Cor:FirstTopExact} and show that this map cannot be $(\Z/2)^k$-equivariantly extended to $Z$.

Let $\mathcal{S}$ be a subgroup of $\sym_k^\pm$.
An extension to the $N_2$-skeleton $Z^{(N_2)}$ can be $\mathcal{S}$-equivariantly extended to $Z$ if and only if the obstruction cocycle
\[
    \mathfrak{o}(g) \in \mathcal{C}_{\mathcal{S}}^{N_2 + 1}\big(Z, X_{d,k}^{>1}; \pi_{N_2}(S(W_k \oplus U_k^{\oplus j}))\big)
\]
is zero. If this cochain is not a coboundary, i.e.
\[
    0 \neq [\mathfrak{o}(g)] \in \mathcal{H}_{\mathcal{S}}^{N_2 + 1}\big(Z, X_{d,k}^{>1}; \pi_{N_2}(S(W_k \oplus U_k^{\oplus j}))\big),
\]
then this map cannot be extended to $Z$ independent of the choice on the $N_2$-skeleton.
In general it can be difficult to determine the obstruction cochain and its cohomology class.
However, the parity of the cochain can be determined by counting equiparting matrices for a suitable choice of the cell $\theta$ and the measures $\mathcal{M}$.
Even better, cochains of odd parity do not vanish in this case:

\begin{proposition}\label{Prop:SecTop}
    Let $\mathcal{S}$ be a subgroup of $\sym_k^\pm$.
    Let $\theta$ be an $(N_2 + 1)$-cell of $X_{d,k}$ such that the $N_2$-faces $\eta_1, \dots, \eta_{n}, \zeta_1,\dots,\zeta_n$ in the boundary of $\theta$ can be grouped in pairs such that $\eta_i$ and $\zeta_i$ are in one $\mathcal{S}$-orbit for each $i = 1,\dots,n$.

    Suppose $\mathcal{M}$ is a collection of measures such that $0$ is not in the image of $\psi_\mathcal{M}$ restricted to the boundary of $\theta$ and that the number of $0$s in the image of the interior of $\theta$ is finite and odd, then $\Delta(j,k) \leq d$.
\end{proposition}
\begin{proof}
    Denote by $e_\theta$ the element in the cellular chain group $\mathcal{C}_{N_2 + 1}(Z, X_{d,k}^{>1})$ corresponding to $\theta$.
    The image of the boundary of $\theta$ under $\psi_\mathcal{M}$ does not contain zero.
    Then by~\cite[Section~2.6]{BFHZ2016}
    $\mathfrak{o}(g)(e_\theta)$ is the same as the number of $x \in \relint \theta$ with $\psi_\mathcal{M}(x) = 0$ modulo $2$.

    We proceed as in~\cite[Proof of Theorem~1.4]{BFHZ2016}.
    Suppose there exists a cochain
    \[
        \mathfrak{h} \in \mathcal{C}_{\mathcal{S}}^{N_2}\big(Z, X_{d,k}^{>1}; \pi_{N_2}(S(W_k \oplus U_k^{\oplus j}))\big)
    \]
    such that $\delta \mathfrak{h} = \mathfrak{o}(g)$, then in particular
    \[
        \mathfrak{o}(g)(e_\theta) = \left(\delta\mathfrak{h}\right)(e_\theta) = \mathfrak{h}(\eta_1 + \dots + \eta_n + \zeta_1 + \dots + \zeta_n).
    \]
    However, it holds that $\mathfrak{h}(\eta_i) = \pm \mathfrak{o}(\zeta_i)$ for all $i = 1,\dots,n$.
    This is a contradiction and hence the obstruction cocycle is not a coboundary and the map cannot be $\mathcal{S}$-equivariantly extended to $Z$.
    By Proposition~\ref{Prop:FirstTop} this implies that~$\Delta(j,k) \leq d$.
\end{proof}

Now we will describe a suitable cell $\theta$.
Let $d \geq 1,\, j \geq 1$ and $k \geq 2$ be integers such that $dk \geq (2^k -1)j$.
Let further $\ell_1 \geq \dots \geq \ell_k \geq 0$ be integers with
$\ell_1 + \dots + \ell_k = dk - (2^k - 1)j$.

Consider the cell
\[
    \theta = D_{\ell_1 + 1, \ell_2 + 1, \dots, \ell_k + 1}^{+,+,\dots,+}(1,2,\dots,k).
\]
This cell has dimension
\[
    (d + 2)k - k - (\ell_1 + \dots + \ell_k) - 1 = k + (2^k - 1)j - 1 = N_2 + 1
\]
as desired.

By definition, $\theta$ parametrizes all arrangements $\mathcal{H} = (H_1,\dots,H_k)$ of $k$ linear oriented hyperplanes in $\mathbb{R}^{d+1}$ with normal vectors $x_1,\dots,x_k$ such that
$x_i$ has the first $\ell_i$ coordinates zero and coordinate $\ell_i + 1$ strictly greater than $x_{i-1}$, where $x_{0} := 0$.

\begin{lemma}
    The boundary of $\theta$ can be grouped in pairs of $\sym_k^\pm$-orbit elements.
    If farther $\ell_1,\dots,\ell_k$ are pairwise distinct, the boundary can be grouped in paris of $(\Z/2)^k$-orbit elements.
\end{lemma}
\begin{proof}
    The cells in the boundary are obtained by introducing one of the following equalities:
    \[
        0 = x_{\ell_1 + 1, 1}, \quad x_{\ell_2 + 1, 1} = x_{\ell_2 + 1, 2}, \quad \dots, \quad x_{\ell_k + 1, k-1} = x_{\ell_k + 1, k}.
    \]
    \begin{enumerate}[(A)]
        \item The equality $0 = x_{\ell_1 + 1, 1}$ induces cells:
            \[
                \eta_1 := D_{\ell_1 + 2, \ell_2 + 1, \dots, \ell_k + 1}^{+,+,\dots,+}(1,2,\dots,k), \quad
                \zeta_1 := D_{\ell_1 + 2, \ell_2 + 1, \dots, \ell_k + 1}^{-,+,\dots,+}(1,2,\dots,k)
            \]
            that are related, as sets via $\zeta_1 = \varepsilon_1 \cdot \eta_1$ ($\varepsilon_1$ is the group element flipping the normal of the first hyperplane).
        \item Let $b \in \{2, \dots, k\}$ and suppose $\ell_{b-1} \geq \ell_b + 1$.
            Let $a$ be minimal such that $\ell_a \leq \ell_{b} + 1$ (possibly $a = b$).
            For the relative interior of $\theta$ it holds that
            \[
                0 <_{\ell_b + 2} x_a <_{\ell_b + 2} x_{a + 1} <_{\ell_b + 2} \dots <_{\ell_b + 2} x_{b-1} <_{\ell_b + 1} < x_b
            \]
            and $a$ is minimal with this property.
            When introducing the equality $x_{\ell_b + 1, b} = x_{\ell_b + 1, b-1}$ there are $b - a + 1$ possible positions for the $b$-th hyperplane, each position with both signs.
            This equality induces the cells
            \begin{align*}
                \eta_{b, b} &:= D_{\ell_1 + 1, \dots, \ell_{b - 1} + 1, \ell_{b} + 2, \ell_{b + 1} + 1, \dots, \ell_{k} + 1}^{+, +, \dots, + }
                (1, \dots, k) \\
                \zeta_{b, b} &:= \varepsilon_b \cdot \eta_{b, b},\\
                \eta_{b-1, b} &:= \tau_{b-1, b} \cdot \eta_{b, b},\\
                \zeta_{b-1, b} &:= \varepsilon_{b-1} \cdot \tau_{b-1, b} \cdot \eta_{b, b} = \varepsilon_{b-1} \cdot \eta_{b-1, b},\\
                \eta_{b-2, b} &:= \tau_{b-2, b-1} \cdot \tau_{b-1, b} \cdot \eta_{b, b},\\
                \zeta_{b-2, b} &:= \varepsilon_{b-2} \cdot \tau_{b-2, b-1} \cdot \tau_{b-1, b} \cdot \eta_{b, b} = \varepsilon_{b-2} \cdot \eta_{b-2, b},\\
                \dots &\\
                \eta_{a, b} &:= \tau_{a, a + 1} \cdot \tau_{a + 1, a + 2}\cdot \dots \cdot \tau_{b-1, b} \cdot \eta_{b, b},\\
                \zeta_{a, b} &:= \varepsilon_{a} \cdot \tau_{a, a + 1} \cdot \tau_{a + 1, a + 2}\cdot \dots \cdot \tau_{b-1, b} \cdot \eta_{b, b} = \varepsilon_{a} \cdot \eta_{a, b}.\\
            \end{align*}
            $\tau_{c,d}$ swaps the hyperplanes $c$ and $d$.
        \item Let $b \in {2, \dots, k}$ and suppose $\ell_{b-1} = \ell_b$, in particular the $(\ell_1,\dots,\ell_k)$ are not pairwise distinct.
            The equality $x_{\ell_b + 1, b} = x_{\ell_b + 1, b-1}$ induces two cells
            \begin{align*}
                \eta_{b, b} &:= D_{\ell_1 + 1, \dots, \ell_{b - 1} + 1, \ell_{b} + 2, \ell_{b + 1} + 1, \dots, \ell_{k + 1}}^{+, +, \dots, + }(1,\dots,k)\\
                \zeta_{b-1, b} &:= \tau_{b-1, b} \cdot \eta_{b, b}.
            \end{align*}
    \end{enumerate}
\end{proof}

Consider the binomial moment curve
\[
    \gamma \colon \mathbb{R} \to \mathbb{R}^{d}, \quad t \mapsto \left(t, \binom{t}{2}, \binom{t}{3},\dots, \binom{t}{d}\right)
\]
and points
\[
    q_1 := \gamma(0), \quad q_2 := \gamma(1), \quad \dots, \quad q_{\ell_1 + 1} = \gamma(\ell_1).
\]
We obtain the following parametrization:

\begin{lemma}\label{Lem:Theta}
    The relative closure of the cell $\theta$ parametrizes those $(v_1,\dots,v_k)$ in $(S^d)^k$ corresponding to affine oriented hyperplanes
    $\mathcal{H} = (H_1,\dots,H_k)$ in $\mathbb{R}^d$ such that
    \begin{itemize}
        \item $\{q_1,\dots, q_{\ell_i}\} \subset H_i$ for $i = 1,\dots, k$,
        \item if $\ell_i = \ell_{i + 1}$ for any $i = 1,\dots, k-1$, then $v_{\ell_i + 1, i} \leq v_{\ell_{i+1} + 1, i + 1}$, where $v_{\ell_i + 1, i}$ is the $(\ell_i + 1)$th-coordinate of the unit normal vector of the hyperplane $H_i$,
        \item $v_{\ell_i + 1, i} \geq 0$ for any $i =1,\dots, k$,
    \end{itemize}
    and for relative open cell $\theta$ additionally
    \begin{itemize}
        \item $q_{\ell_i + 1} \not\in H_i$ for $i=1,\dots,k$,
        \item if $\ell_i = \ell_{i + 1}$ for any $i = 1,\dots, k-1$, then $v_{\ell_i + 1, i} < v_{\ell_{i+1} + 1, i + 1}$,
        \item $v_{\ell_i + 1, i} > 0$ for any $i =1,\dots, k$.
    \end{itemize}
\end{lemma}
    By parametrize we mean the restriction of $X_{d,k}$ to the embedding of $(S^d)^k$:
    \[
        \left\{(\lambda_1 v_1 ,\dots \lambda_k v_k) \in X_{d,k} \colon \lambda_1 = \dots = \lambda_k = \frac{1}{k}  \right \}.
    \]
    If $\psi_{\mathcal{M}}(x) = 0$, then $x$ lies in the image of the embedding.
\begin{proof}
    This proof is analogous to \cite[Lem.~3.13]{BFHZ2016}:

    By Definition~\ref{def:Cell} the relatively open cell $\theta$ corresponds to hyperplanes with normals $v_1,\dots,v_k$ such that
    \[
        0 <_{\ell_1 + 1} v_1 <_{\ell_2 + 1} v_2 <_{\ell_{3} + 1} \dots <_{\ell_{k} + 1} v_k.
    \]
    As $\ell_1 \geq \ell_2 \geq \dots \geq \ell_k$, by induction on $i$ we see that this is equivalent to
    \[
        0 = v_{i,1}, \quad 0 = v_{i,2}, \quad \dots, \quad 0 = v_{i, \ell_i}, \quad 0 \leq v_{i-1, \ell_i + 1} < v_{i, \ell_i +1}.
    \]
    for each $i = 1,\dots,k$, where $v_0 := 0$.
    The relative closure is obtained by allowing the strict inequality to be non-strict.

    Observe that a hyperplane with normal $v$, that is zero in the first $\ell$ entries contains $q_{\ell + 1}$ if and only if the $\ell + 1$th entry is zero as well.
    The characterization of the relatively open cell $\theta$ and its closure follow.
\end{proof}

We are now ready to use Theorem~\ref{Thm:OddEquipartingMatrices}:
\begin{proposition}\label{Prop:MainReduction} \
    \begin{enumerate}
        \item $\Delta(j, k) \leq I(j,k)$.\label{Prop:MainReductionExact}
        \item $\Delta(j, k) \leq E(j,k)$.\label{Prop:MainReductionEquivariant}
    \end{enumerate}
\end{proposition}
Along with Theorem~\ref{Thm:OddEquipartingMatrices} this proves
Theorem~\ref{Thm:TopMainStatement}.
The part \ref{Prop:MainReductionExact} is strictly weaker, but uses only the $(\Z/2)^k$-action.
Its purpose is to explain why $I(j,k)$ coincides with the bound provided in~\cite{Mani}.
\begin{proof}
    \begin{enumerate}
        \item[(2)]
            Suppose that $d = E(j,k)$.
            In this case there are $d = \mathfrak{c}_1 \geq \mathfrak{c}_2 \geq \dots \geq \mathfrak{c}_k \geq 0$ with
            \[
                \mathfrak{c}_1 + \dots + \mathfrak{c}_k = j(2^k-1).
            \]
            such that the number of non-equivalent $(\mathfrak{c}_1,\dots,\mathfrak{c}_k)$-equiparting matrices is odd.

        \item[(1)]
            If $d = I(j,k)$ there are such $\mathfrak{c}_1,\dots,\mathfrak{c}_k$ as above which are additionally pairwise distinct.
            Note that in this case $\stabsym$ is trivial and $\stabsym^\pm = (\Z/2)^k$ and
            non-equivalent $(\mathfrak{c}_1,\dots,\mathfrak{c}_k)$-equiparting matrices are the same as non-isomorphic ones.
    \end{enumerate}

    Either way, define
    \[
        \ell_1 := d - \mathfrak{c}_1, \quad \dots, \quad \ell_k := d - \mathfrak{c}_k.
    \]
    Consider the cell
    \begin{align*}
        \theta &= D_{\ell_1 + 1, \ell_2 + 1, \dots, \ell_k + 1}^{+,+,\dots,+}(1,2,\dots,k) \\
        &= D_{d - \mathfrak{c}_1 + 1, d - \mathfrak{c}_2 + 1, \dots, d - \mathfrak{c}_k + 1}^{+,+,\dots,+}(1,2,\dots,k).
    \end{align*}
    In order to use Proposition~\ref{Prop:SecTop} it remains to find measures $\mathcal{M}$ such that the image of $\psi_\mathcal{M}$
    \begin{itemize}
        \item does not contain $0$ restricted to the boundary of $\theta$,
        \item contains an odd number of $0$s restricted to the relative interior of $\theta$.
    \end{itemize}
    We proceed as in~\cite[Lemma~4.1, Lemma~4.2]{BFHZ2016}.
    However, in \cite{BFHZ2016} it was omitted to show that the $0$s are in the relative open cell $\theta$ (and not just in the closure as was done in \cite[Thm.~1.3]{BFHZ2016}).
    Here we provide a proof for completeness.

    A measure on $\mathbb{R}$ induces a measure on $\mathbb{R}^{d}$ via the map $\gamma \colon \mathbb{R} \to \mathbb{R}^{d}$.
    For points $p_0 < p_1 < \dots < p_{2^k}$ we construct a Borel measure $\mu_{p_0,\dots,p_{2^k}}$ on $\mathbb{R}^d$ induced by $\mu$ on $\mathbb{R}$ as follows
    \begin{itemize}
        \item $\mu([p_i, p_{i+1}]) = 1$ for $i=0,\dots,2^k-1$ and $1$ is uniformly distributed on this interval,
        \item $\mu( (-\infty, p_0]) = 0$,
        \item $\mu( [p_{2^k}, \infty) ) = 0$.
    \end{itemize}
    The measure $\mu_{p_0,\dots,p_{2^k}}$ can only be equiparted into $2^k$ pieces by $k$ hyperplanes, if at least $2^k -1$ points in $\gamma( (p_0, p_{2^k}) )$ are contained in one of the hyperplanes $H_1,\dots, H_k$.
    More than $2^k -1$ points are needed, unless the hyperplanes contain exactly the points $\gamma(p_1),\dots,\gamma(p_{2^k -1})$.
    If we encode the intervals $[p_0, p_1], \dots, [p_{2^k -1}, p_{2^k}]$ with vectors in $\{0, 1\}^k$ according to the region,
    an equipartition of $\mu_{p_0,\dots,p_{2^k}}$ by $2^k-1$ points corresponds exactly to a $k$-bit Gray code.

    Overall we choose points
    \[
        \ell_1
        \quad = \quad
        p_{1, 0} < \dots < p_{1, 2^k}
        \quad = \quad
        p_{2, 0} < \dots < p_{2, 2^k}
        \quad = \quad
        \dots
        \quad = \quad
        p_{j, 0} < \dots < p_{j, 2^k}
    \]
    and obtain measures
    \[
        \mathcal{M} = \left(\mu_{p_{1,0}, \dots, p_{1,2^k}}, \, \dots,\,  \mu_{p_{j, 0}, \dots, p_{j, 2^k}}\right).
    \]
    An equipartition of all measures simultanously can only be obtained with at least $j(2^k -1)$ intersection points of the $H_1,\dots,H_k$ with $\gamma( (p_{1,0}, p_{j,2^k}))$.
    The equipartitions with exactly $j(2^k -1)$ intersection points are encoded by equiparting matrices.

    According to Lemma~\ref{Lem:Theta}, for an arrangement $(H_1,\dots,H_k)$ in the relative closure of $\theta$,
    the hyperplane $H_i$ contains the points $q_1,\dots,q_{\ell}$ on $\gamma(\mathbb{R})$ for each $i = 1,\dots,k$.
    Each hyperplane can have at most $d$ intersection points with $\gamma(\mathbb{R})$ and
    \[
        dk = \ell_1 + \dots + \ell_k + j(2^k -1).
    \]
    So, hyperplane arrangements in the relative closure of $\theta$ have at most $j(2^k-1)$ intersection points with $\gamma( (p_{1,0},p_{j,2^k}))$.
    We conclude that arrangments in the $\stabsym^\pm$-orbit of the relative closure of $\theta$ that equipart $\mathcal{M}$ into $2^k$ pieces are in one-to-one correspondence with $(\mathfrak{c}_1,\dots,\mathfrak{c}_k)$-equiparting matrices.

    The number of non-equivalent $(\mathfrak{c}_1,\dots,\mathfrak{c}_k)$-equiparting matrices is odd by assumption.
    If each of those arrangements has a representative in the relatively open cell $\theta$ and not just in the relative closure,
    we are done, as $\sym_k^{\pm}$ acts free on the $\theta$-orbit.

    Clearly, for $i=1,\dots,k$ no hyperplane $H_i$ in such an arrangement may contain the point $q_{\ell_i + 1}$ and we conclude $v_{\ell_i + 1, i} \neq 0$.
    According to Lemma~\ref{Lem:NumberTheory}, we can choose the points $p_{1,0} < \dots < p_{j, 2^k}$ such that
    $\ell_i = \ell_{i+1}$ for any $i \in \{1,\dots,k-1\}$ implies $v_{\ell_i + 1, i} \neq v_{\ell_{i + 1} + 1, i+1}$
    and then indeed the arrangement lies in the relatively open cell $\theta$ by Lemma~\ref{Lem:Theta}.
\end{proof}

\begin{lemma}\label{Lem:NumberTheory}
    Let $d > \ell_1 > 0$ and $n \geq 2(d - \ell_1)$ be integers.
    There exist $\ell_1 < \pi_1 < \dots < \pi_n$ in $\mathbb{R}$ such that for any two oriented hyperplanes $H_v, H_w$ parametrized by $v,w \in S^d$
    and $\ell \leq \ell_1$
    with the following properties
    \begin{enumerate}
        \item $v_i = w_i = 0$ for $i = 1,\dots,\ell$, i.e. $H_v, H_w$ contain $q_1,\dots, q_\ell$,
        \item $v_{\ell + 1} \neq 0 \neq w_{\ell + 1}$, i.e. $H_v, H_w$ do not contain $q_{\ell + 1}$,
        \item the sets
            \[
                H_v \cap \left\{\gamma(\pi_1), \dots \gamma(\pi_n)\right\}, \quad
                H_w \cap \left\{\gamma(\pi_1), \dots \gamma(\pi_n)\right\}
            \]
            are disjoint and both of cardinality $d - \ell$,
    \end{enumerate}
    it holds that $v_{\ell + 1} \neq w_{\ell + 1}$.
\end{lemma}
\begin{proof}
    Let $\pi_1,\dots,\pi_n$ be algebraically independent with $\ell_1 < \pi_1 < \dots < \pi_n$.
    Suppose
    \[
        H_v \cap \left\{\gamma(\pi_1), \dots \gamma(\pi_n)\right\} = \left\{\gamma(\pi_{I_1}), \dots \gamma(\pi_{I_{d-\ell}})\right\}, \quad
            H_w \cap \left\{\gamma(\pi_1), \dots \gamma(\pi_n)\right\} = \left\{\gamma(\pi_{J_1}), \dots \gamma(\pi_{J_{d-\ell}})\right\}.
    \]
    This means that $\pi_{I_1},\dots, \pi_{I_{d-\ell}}$ are roots of the polynomial
    \[
        1 v_1 + t v_2 + \binom{t}{2} v_3 + \dots + \binom{t}{d} v_{d+1}
    \]
    in $\mathbb{Q}(v_1,\dots,v_{d+1})[t]$.
    As $v_1 = \dots = v_\ell = 0$ and $v_1^2 + \dots + v_{d+1}^2 = 1$,
    they are algebraic over $\mathbb{Q}(v_{\ell + 1},\dots,v_d)$.

    Likewise, $\pi_{J_1}, \dots, \pi_{J_{d-\ell}}$ are algebraic over $\mathbb{Q}(w_{\ell + 1},\dots, w_{d})$.
    By assumptions the transcendence degree of $\mathbb{Q}(\pi_{I_1},\dots,\pi_{I_{d-\ell}}, \pi_{J_1},\dots\pi_{J_{d-\ell}})$ is $2(d - \ell)$.
    This implies that $v_{\ell + 1}$ and $w_{\ell + 1}$ are algebraically independent and in particular distinct.
\end{proof}



\section*{Acknowledgements}
I would like to thank Pavle Blagojević for proposing to study~\cite{BFHZ2016},
for many valuable discussions and for providing relevant references.
I would further thank Florian Frick for his help and questions that helped to shape this paper.

\end{document}